\documentclass[12pt, reqno, a4paper]{amsart}

\usepackage{amssymb}
\usepackage{hyperref}
\usepackage[pdftex]{lscape}

\usepackage[english]{babel}

\theoremstyle{plain}
\newtheorem{theorem}{Theorem}[section]
\newtheorem{lemma}[theorem]{Lemma}
\newtheorem{proposition}[theorem]{Proposition}
\newtheorem{corollary}[theorem]{Corollary}

\theoremstyle{remark}
\newtheorem{remark}[theorem]{Remark}

\theoremstyle{definition}
\newtheorem{example}[theorem]{Example}

\numberwithin{equation}{section}

\DeclareMathOperator{\Id}{Id}

\DeclareMathOperator{\id}{id}
\DeclareMathOperator{\ch}{char}

\DeclareMathOperator{\UT}{UT}

\DeclareMathOperator{\supp}{supp}

\DeclareMathOperator{\End}{End}

\DeclareMathOperator{\sign}{sign}

\DeclareMathOperator{\Hom}{Hom}

\DeclareMathOperator{\PIexp}{PIexp}
\newcommand{\hatotimes}{\mathbin{\widehat{\otimes}}}

\begin{document}

\title{On $H$-simple not necessarily associative algebras}

\author{A.\,S.~Gordienko}
\address{Vrije Universiteit Brussel, Belgium}
\email{alexey.gordienko@vub.ac.be}

\keywords{Polynomial identity, $H$-module algebra, generalized $H$-action, codimension, PI-exponent, non-associative algebra, grading, semigroup, free-forgetful adjunction.}

\begin{abstract} An algebra $A$ with a generalized $H$-action is a generalization of an $H$-module algebra where $H$ is just an associative algebra with $1$ and a relaxed compatibility condition
between the multiplication in $A$ and the $H$-action on $A$ holds.
At first glance this notion may appear too general, however it enables to work with algebras endowed with various kinds of additional structures (e.g. (co)module algebras over Hopf algebras, graded algebras, algebras with an action of a (semi)group by (anti)endomorphisms). This approach proves to be especially fruitful in the theory of polynomial identities. We show that if $A$ is a finite dimensional (not necessarily associative) algebra over a field of characteristic $0$ and $A$ is simple
with respect to a generalized $H$-action,
then there exists $\lim_{n\to\infty}\sqrt[n]{c_n^H(A)} \in \mathbb R_+$
where $\left(c_n^H(A)\right)_{n=1}^\infty$ is the sequence of codimensions of polynomial $H$-identities of $A$. 
In particular, if $A$ is a finite dimensional (not necessarily group graded) 
graded-simple algebra, then there exists $\lim_{n\to\infty}\sqrt[n]{c_n^{\mathrm{gr}}(A)} \in \mathbb R_+$
where $\left(c_n^{\mathrm{gr}}(A)\right)_{n=1}^\infty$ is the sequence of codimensions of graded polynomial identities of $A$. 
In addition, we study the free-forgetful adjunctions corresponding to
(not necessarily group) gradings and generalized $H$-actions.
\end{abstract}

\subjclass[2010]{Primary 17A30; Secondary 16R10, 16R50, 16T05, 17A36, 17A50,
18A40, 20C30.}

\thanks{Supported by Fonds Wetenschappelijk Onderzoek~--- Vlaanderen post doctoral fellowship (Belgium).}

\maketitle

\section{Introduction}

Study of polynomial identities in algebras is an important aspect of study of algebras themselves. It turns out that the asymptotic behaviour of numeric characteristics of polynomial identities
of an algebra is tightly related to the structure of the algebra~\cite{ZaiGia, ZaiLie}.

In 1980s, S.\,A.~Amitsur conjectured that if an associative algebra $A$ over a field of characteristic $0$ satisfies a nontrivial
polynomial identity, then there exists an integer \textit{PI-exponent} $\lim_{n\to\infty}\sqrt[n]{c_n(A)}$ where $c_n(A)$ is the codimension sequence of ordinary polynomial identities of $A$.
(See the definition of $c_n(A)$ in Remark~\ref{RemarkOrdinaryCodim} below.)
The original Amitsur conjecture was proved by A.~Giambruno and M.\,V.~Zaicev~\cite{ZaiGiaAmitsur} in 1999.
Its analog for finite dimensional Lie algebras was proved by M.\,V.~Zaicev~\cite{ZaiLie}
in 2002. In 2011
A.~Giambruno, I.\,P.~Shestakov and M.\,V.~Zaicev proved
the analog of the conjecture for finite dimensional
Jordan and alternative algebras~\cite{GiaSheZai}.

In general, the analog of Amitsur's conjecture for arbitrary non-associative
algebras and even for infinite dimensional Lie algebras is wrong.
First, the codimension growth can be overexponential~\cite{Volichenko}.
Second, the exponent of the codimension growth can be non-integer~\cite{GiaMishZaiCodimGrFun, ZaiMishchFracPI, VerZaiMishch}. Third, in 2014 M.\,V.~Zaicev constructed an example
of an infinite dimensional non-associative algebra $A$ for which
$\mathop{\underline\lim}_{n\to\infty}\sqrt[n]{c_n(A)}=1$
and  $\mathop{\overline\lim}_{n\to\infty}\sqrt[n]{c_n(A)} > 1$~\cite{ZaicevPIexpDoesNotExist}.

Algebras endowed with an additional structure, e.g. a grading, an action of a group, a Lie algebra or a Hopf algebra, find their applications in many areas of mathematics and physics.
 Gradings on simple Lie and associative algebras
have been studied extensively~\cite{BahturinSehgalZaicevGrAssoc, BahZaicAllGradings, BahturinZaicevSegal, ElduqueKochetov, PZ89, HPP98, HPP00}. For algebras with an additional structure, it is natural to consider the corresponding polynomial identities.

E.~Aljadeff, A.~Giambruno, and D.~La~Mattina~\cite{AljaGia, AljaGiaLa, GiaLa} proved that if an  associative PI-algebra is graded by a finite group, then the graded PI-exponent exists and it is an integer,
i.e. the graded analog of Amitsur's conjecture holds for
such algebras.
The same is true for finite dimensional associative and Lie algebras
graded by arbitrary groups~\cite[Theorem~3]{ASGordienko9}, \cite[Theorem~1]{ASGordienko5}.
If $H$ is a finite dimensional semisimple Hopf algebra,
then the codimensions of polynomial $H$-identities of any finite dimensional $H$-module
associative or Lie algebra satisfy the analog of Amitsur's conjecture too~\cite[Theorem~3]{ASGordienko3}, \cite[Theorem~7]{ASGordienko5}.
If an algebra is graded by a semigroup, then its graded PI-exponent
can be non-integer even if the algebra itself is finite dimensional and associative~\cite[Theorem 5]{ASGordienko13} (see also~\cite{GordienkoJanssensJespers}).

In order to embrace the cases when an algebra is graded by a semigroup or an infinite group,
or a group is acting on an algebra not only by automorphisms, but by anti-automorphisms too, it is
useful to consider more general additional structures.
Although polynomial identities in algebras endowed with an arbitrary set of multilinear operations were considered in the literature (see e.g.~\cite[Section~49.2]{Razmyslov}),
it turns out that so-called generalized $H$-actions
(see the definition in Section~\ref{SectionHmodGen})
where $H$ is an arbitrary associative algebra with $1$ are the most adequate for our purposes.
 Probably, the first who used such actions and studied polynomial $H$-identities was Allan Berele~\cite[remark before Theorem~15]{BereleHopf} in 1996 (see also the paper of Yu.\,A.~Bahturin and V.\,V.~Linchenko~\cite{BahturinLinchenko}).
The example constructed in~\cite[Theorem 5]{ASGordienko13}
shows that for generalized $H$-actions the exponent of the $H$-codimension
growth can be non-integer even for finite dimensional $H$-simple associative algebras. Therefore, the natural question arises as to whether $H$-PI-exponent exists at all, at least in the case when the algebra is $H$-simple.

In 2012 A.~Giambruno and M.\,V.~Zaicev
proved the existence of the ordinary PI-exponent
for any finite dimensional simple algebra not necessarily associative~\cite[Theorem 3]{ZaiGiaFinDimSuperAlgebras}.
Recently D.~Repov\v s and M.\,V.~Zaicev proved 
the existence of the graded PI-exponent
for finite dimensional graded-simple algebras graded by commutative semigroups~\cite[Theorem 2]{ZaicevRepovs}. 

In the present article we combine A.~Giambruno and M.\,V.~Zaicev's
techniques with the techniques of generalized $H$-actions
and show that for any finite dimensional $H$-simple algebra with a generalized $H$-action
 there exists an $H$-PI-exponent (Theorem~\ref{TheoremHSimpleHPIexpHNAssoc}). This enables to prove (see Corollary~\ref{CorollaryGradedExistsExponent}) the existence
of the graded PI-exponent for any finite dimensional graded-simple algebra graded in a very general sense (not necessary by a semigroup, see the precise
definition of such a grading in Example~\ref{ExampleGr}).
Note that the notion of an $H$-simple algebra is much wider than the notion of a simple algebra since, e.g., an $H$-simple associative or Lie algebra is not even necessarily semisimple.

One of the important steps in the proof of Theorem~\ref{TheoremHSimpleHPIexpHNAssoc} is Theorem~\ref{TheoremUpperBoundColengthHNAssoc}
where we show that $H$-colengths of a finite dimensional algebra with a generalized $H$-action are polynomially bounded (see Corollary~\ref{CorollaryUpperBoundColengthGraded} 
 for the analog in the graded case).

Polynomial $H$-identities
and graded polynomial identities
are elements of the algebras $F\lbrace X | H \rbrace$
and $F\lbrace X^{T\text{-}\mathrm{gr}} \rbrace$ defined 
in Sections~\ref{SectionHPI} and~\ref{SectionGradedPI}, respectively.
In fact, if $H$ is an arbitrary unital associative algebra
and $T$ is an arbitrary set, then neither $F\lbrace X | H \rbrace$
is an algebra with a generalized $H$-action, nor $F\lbrace X^{T\text{-}\mathrm{gr}} \rbrace$ is a $T$-graded algebra (which, however,
does not prevent studying polynomial $H$-identities in algebras
with generalized $H$-actions and graded polynomial identities
in $T$-graded algebras at all).
In Section~\ref{SectionGrHAdjunction} we show that if we enlarge
the categories of algebras in a proper way, then both $F\lbrace X | H \rbrace$ and $F\lbrace X^{T\text{-}\mathrm{gr}} \rbrace$ will correspond to free-forgetful adjunctions.

 \section{Algebras with a generalized $H$-action}\label{SectionHmodGen}
 
  Let $H$ be an arbitrary associative algebra with $1$ over a field $F$.
We say that a (not necessarily associative) algebra $A$ is an algebra with a \textit{generalized $H$-action}
if $A$ is a left $H$-module
and for every $h \in H$ there exist some $k\in \mathbb N$ and some $h'_i, h''_i, h'''_i, h''''_i \in H$, $1\leqslant i \leqslant k$,
such that 
\begin{equation}\label{EqGeneralizedHopf}
h(ab)=\sum_{i=1}^k\bigl((h'_i a)(h''_i b) + (h'''_i b)(h''''_i a)\bigr) \text{ for all } a,b \in A.
\end{equation}

Equivalently, there exist linear maps $\Delta, \Theta \colon H \to H\otimes H$ (not necessarily coassociative)
such that 
$$ h(ab)=\sum\bigl((h_{(1)} a)(h_{(2)} b) + (h_{[1]} b)(h_{[2]} a)\bigr) \text{ for all } a,b \in A.$$ (Here we use the notation $\Delta(h)= \sum h_{(1)} \otimes h_{(2)}$ and $\Theta(h)= \sum  h_{[1]} \otimes h_{[2]}$.)

\begin{example}\label{ExampleHmodule} 
An algebra $A$
over a field $F$
is a \textit{(left) $H$-module algebra}
for some Hopf algebra $H$
if $A$ is endowed with a structure of a (left) $H$-module such that
$h(ab)=(h_{(1)}a)(h_{(2)}b)$
for all $h \in H$, $a,b \in A$. Here we use Sweedler's notation
$\Delta h = \sum h_{(1)} \otimes h_{(2)}$ where $\Delta$ is the comultiplication
in $H$ and the symbol of the sum is omitted. If $A$ is an $H$-module algebra,
then $A$ is an algebra with a generalized $H$-action.
\end{example}

\begin{example}\label{ExampleIdFT}
Recall that if $T$ is a semigroup, then the \textit{semigroup algebra} $FT$ over a field is the vector space with the formal basis $(t)_{t\in T}$ and the multiplication induced by the one in $T$.
Let $A$ be an associative algebra with an action 
of a semigroup $T$ by endomorphisms and anti-endomorphisms. Then $A$ is an algebra with
 a generalized $FT$-action.
\end{example}

\begin{example}\label{ExampleGr}
Let $A=\bigoplus_{t\in T} A^{(t)}$ be a \textit{graded} algebra for some set of indices $T$, i.e. for every $s,t \in T$ there exists $r\in T$
such that $A^{(s)}A^{(t)}\subseteq A^{(r)}$. Denote this grading by $\Gamma$. Note that $\Gamma$ defines on $T$
a partial operation $\star$ with the domain $T_0:=\lbrace (s,t) \mid A^{(s)}A^{(t)} \ne 0 \rbrace$ by $s\star t := r$.
Consider the algebra $F^T$ of all functions from $T$ to $F$ with pointwise operations.
Then $F^T$ acts on $A$ naturally: $ha = h(t)a$ for all $a\in A^{(t)}$.
Denote by $\zeta \colon F^T \to \End_F(A)$
the corresponding homomorphism.
Let $h_t(s):=\left\lbrace\begin{smallmatrix} 1 & \text{if} & s=t,\\ 0 & \text{if} & s\ne t.\end{smallmatrix} \right.$
If the \textit{support} $$\supp \Gamma := \lbrace t\in T \mid A^{(t)}\ne 0\rbrace$$ of $\Gamma$ is finite, $T_0$ is finite too and we have
\begin{equation}\label{EqIdentityHFiniteSupp}h_r(ab)=\sum_{\substack{(s,t)\in T_0,\\ r=s\star t}}
h_s(a)h_t(b). \end{equation}
(Since the expression is linear in $a$ and $b$, it is sufficient to check it only for homogeneous $a,b$.) In this case $(h_t)_{t\in \supp \Gamma}$
generates $F^T$ modulo $\ker \zeta$ as an $F$-vector space and, by the linearity, \eqref{EqIdentityHFiniteSupp} implies~(\ref{EqGeneralizedHopf}) for every $h\in F^T$. Therefore, in the case of a finite support of the grading, $A$ is an algebra with a generalized $F^T$-action.
\end{example}

Let $A$ be an algebra with a generalized $H$-action for some associative algebra $H$ with $1$ over a field $F$. We say that a subspace $V\subseteq A$ is  \textit{invariant}
under the $H$-action if $HV=V$, i.e. $V$ is an $H$-submodule.
If $A^2\ne 0$ and $A$ has no non-trivial
two-sided $H$-invariant ideals, we say that $A$ is \textit{$H$-simple}.

\section{Polynomial $H$-identities}\label{SectionHPI}

Let $F$ be a field and let $Y$ be a set.
 Denote by $F \lbrace Y \rbrace$ the absolutely free non-associative algebra
 on the set $Y$, i.e. the algebra of all non-associative polynomials in variables
 from $Y$ and coefficients from the field $F$.
  Then $F \lbrace Y \rbrace = \bigoplus_{n=1}^\infty F \lbrace Y \rbrace^{(n)}$
  where $F \lbrace Y \rbrace^{(n)}$ is the linear span of all monomials of total degree $n$.
Let $H$ be an associative algebra with $1$ over $F$.
  
       Consider the algebra $$F \lbrace Y | H\rbrace := \bigoplus_{n=1}^\infty H^{{}\otimes n} \otimes F \lbrace Y \rbrace^{(n)}$$
   with the multiplication $(u_1 \otimes w_1)(u_2 \otimes w_2):=(u_1 \otimes u_2) \otimes w_1w_2$
   for all $u_1 \in  H^{{}\otimes j}$, $u_2 \in  H^{{}\otimes k}$,
   $w_1 \in F \lbrace Y \rbrace^{(j)}$, $w_2 \in F \lbrace Y \rbrace^{(k)}$.
We use the notation $$y^{h_1}_1
y^{h_2}_2\cdots y^{h_n}_n := (h_1 \otimes h_2 \otimes \cdots \otimes h_n) \otimes y_1
y_2\cdots y_n$$ (the arrangements of brackets on $y_j$ and on $y^{h_j}_j$
are the same). Here $h_1 \otimes h_2 \otimes \cdots \otimes h_n \in H^{{}\otimes n}$,
$y_1, y_2,\dots, y_n \in Y$. In addition, we identify $Y$ with the subset
$\lbrace y^{1_H} \mid y\in Y\rbrace \subset F \lbrace Y | H\rbrace$.

Note that if $(\gamma_\beta)_{\beta \in \Lambda}$ is a basis in $H$, 
then $F \lbrace Y | H\rbrace$ is isomorphic to the absolutely 
free non-associative algebra over $F$ with free formal
generators $y^{\gamma_\beta}$, $\beta \in \Lambda$, $y \in Y$.
We call $F \lbrace Y | H\rbrace$ \textit{the absolutely free non-associative algebra on $Y$
with symbols from $H$}.

Below we consider $F \lbrace X | H \rbrace$ where $X := \lbrace x_1, x_2, x_3, \dots \rbrace$.
The elements of $F \lbrace X | H\rbrace$ are called \textit{$H$-polynomials}.

Let $A$ be an algebra over $F$ with a generalized $H$-action.
Any map $\psi \colon X \to A$ has a unique homomorphic extension $\bar\psi
\colon F \lbrace X | H\rbrace \to A$ such that $\bar\psi(x_i^h)=h\psi(x_i)$
for all $i \in \mathbb N$ and $h \in H$.
 An $H$-polynomial
 $f \in F\lbrace X | H\rbrace$
 is a \textit{polynomial $H$-identity} of $A$ if $\bar\psi(f)=0$
for all maps $\psi \colon X \to A$. In other words, $f(x_1, x_2, \dots, x_n)$
 is an $H$-identity of $A$
if and only if $f(a_1, a_2, \dots, a_n)=0$ for any $a_i \in A$.
 In this case we write $f \equiv 0$.
The set $\Id^{H}(A)$ of all polynomial $H$-identities of $A$ is an ideal of $F\lbrace X | H\rbrace$.

Denote by $W^H_n$, $n\in\mathbb N$, the space of all multilinear non-associative $H$-polynomials
in $x_1, \dots, x_n$, i.e.
$$W^{H}_n = \langle x^{h_1}_{\sigma(1)}
x^{h_2}_{\sigma(2)}\cdots x^{h_n}_{\sigma(n)}
\mid h_i \in H, \sigma\in S_n \rangle_F \subset F \lbrace X | H \rbrace.$$
(We consider all possible arrangements of brackets.)
Then the number $c^H_n(A):=\dim\left(\frac{W^H_n}{W^H_n \cap \Id^H(A)}\right)$
is called the $n$th \textit{codimension of polynomial $H$-identities}
or the $n$th \textit{$H$-codimension} of $A$.
If $f\in W_n^H$, then its image in $\frac{W^H_n}{W^H_n \cap \Id^H(A)}$
is denoted by $\bar f$. The limit $\PIexp^H(A):=\lim_{n\to\infty} \sqrt[n]{c^H_n(A)}$,
if it exists, is called \textit{the $H$-PI-exponent} of $A$.

One of the main tools in the investigation of polynomial
identities is provided by the representation theory of symmetric groups.
 The symmetric group $S_n$  acts
 on the space $\frac {W^H_n}{W^H_{n}
  \cap \Id^H(A)}$
  by permuting the variables.
  If the characteristic of the base field $F$ is zero, 
  then irreducible $FS_n$-modules are described by partitions
  $\lambda=(\lambda_1, \dots, \lambda_s)\vdash n$ and their
  Young diagrams $D_\lambda$.
   The character $\chi^H_n(A)$ of the
  $FS_n$-module $\frac {W^H_n}{W^H_n
   \cap \Id^H(A)}$ is
   called the $n$th
  \textit{cocharacter} of polynomial $H$-identities of $A$.
  We can rewrite it as
  a sum $$\chi^H_n(A)=\sum_{\lambda \vdash n}
   m(A, H, \lambda)\chi(\lambda)$$ of
  irreducible characters $\chi(\lambda)$.
  The number $\ell_n^H(A):=\sum_{\lambda \vdash n}
   m(A, H, \lambda)$ is called the $n$th
  \textit{colength} of polynomial $H$-identities of $A$.
Let  $e_{T_{\lambda}}=a_{T_{\lambda}} b_{T_{\lambda}}$
and
$e^{*}_{T_{\lambda}}=b_{T_{\lambda}} a_{T_{\lambda}}$
where
$a_{T_{\lambda}} = \sum_{\pi \in R_{T_\lambda}} \pi$
and
$b_{T_{\lambda}} = \sum_{\sigma \in C_{T_\lambda}}
 (\sign \sigma) \sigma$,
be the Young symmetrizers corresponding to a Young tableau~$T_\lambda$.
Then $M(\lambda) = FS_n e_{T_\lambda} \cong FS_n e^{*}_{T_\lambda}$
is an irreducible $FS_n$-module corresponding to
 a partition~$\lambda \vdash n$.
  We refer the reader to~\cite{Bahturin, DrenKurs, ZaiGia}
   for an account
  of $S_n$-representations and their applications to polynomial
  identities.
  
  \begin{remark}\label{RemarkOrdinaryCodim} 
 If $A$ is an ordinary algebra, then the ordinary polynomial identities and cocharacters of $A$
can be defined as $H$-identities and $H$-cocharacters
for $H=F$ acting on $A$ trivially:
$W_n := W^F_n$,
$\Id(A):=\Id^F(A)$,
$c_n(A):=c_n^F(A)$,
$\chi_n(A) := \chi^F_n(A)$,
$m(A,\lambda):= m(A,F,\lambda)$,
$\ell_n(A):= \ell_n^F(A)$.
\end{remark}

\begin{remark}
Note that here we do not consider any $H$-action on $F \lbrace Y | H\rbrace$ itself. However $F \lbrace - | H\rbrace$
can be viewed as a free functor if we enlarge the category
of algebras with a generalized $H$-action in a proper way
(see Section~\ref{SectionHAdjunction}).
\end{remark}

\begin{remark}\label{RemarkHActionOnWHn}
The real importance of~\eqref{EqGeneralizedHopf} is hidden in the fact that~\eqref{EqGeneralizedHopf} makes it possible to define a structure of a left $H$-module on the vector space $\frac {W^H_n}{W^H_{n}  \cap \Id^H(A)}$ in the natural way as follows. Every multilinear $H$-polynomial can be regarded as a multilinear function on $A$
taking values in $A$ and $W^H_{n}  \cap \Id^H(A)$ is precisely the kernel
of the corresponding homomorphism $W^H_n \to \Hom_F(A^{{}\otimes n}; A)$ of $FS_n$-modules
where $S_n$ is acting by permuting the arguments.
Hence we have an embedding $\frac {W^H_n}{W^H_{n}  \cap \Id^H(A)} \subseteq \Hom_F(A^{{}\otimes n}; A)$.
The space $\Hom_F(A^{{}\otimes n}; A)$ is a left $H$-module and operators from $H$ commute with 
operators from $S_n$: $(hg)(a_1, \ldots, a_n):= hg(a_1, \ldots, a_n)$ for $h\in H$,
$g\in \Hom_F(A^{{}\otimes n}; A)$ and $a_1, \ldots, a_n \in A$.
Given $f\in W^H_{n}$ and $h\in H$, we can apply~\eqref{EqGeneralizedHopf} several times
and rewrite the function $h\bar f$ as a linear
combination of products of the variables $x_i$ where the operators from $H$ are applied only to the variables $x_i$ themselves and not to their products. In other words, the 
function $h\bar f$ on $A$ can be presented (not necessarily in a unique way) as a
function corresponding to an $H$-polynomial from $W^H_n$.  
As a consequence, $\frac {W^H_n}{W^H_{n}  \cap \Id^H(A)}$ is an $H$-submodule.
 Denote by $f^h$ any $H$-polynomial such that $\overline{f^h}=h\bar f$.
 If $\ch F = 0$,  $f^h \notin \Id^H(A)$ and $FS_n \bar f \cong M(\lambda)$ for some $\lambda \vdash n$, then $FS_n h\bar f$ is a nonzero homomorphic
 image of the irreducible $FS_n$-module $M(\lambda)$. As a consequence, $FS_n h\bar f \cong M(\lambda)$ too.
We will use this property in Section~\ref{SectionHPIexpExistHSimple}.
 \end{remark}

\begin{remark}\label{RemarkHAssNAssCodimTheSame} Suppose $A$ is associative.
One can analogously construct the free associative algebra $F\langle X | H \rangle$ on $X$ with symbols from $H$
(see~\cite[Section 3.1]{ASGordienko3})
and treat polynomial $H$-identities
as elements of an ideal $\Id^H_{\mathrm{assoc}}(A)$ of $F\langle X | H \rangle$. However, the map $x_i^h \mapsto x_i^h$, $i\in \mathbb N$, $h\in H$,
induces an isomorphism $F \lbrace X | H\rbrace/ \Id^H_{\mathrm{assoc}}(A) \cong F\langle X | H \rangle/\Id^H(A)$ of algebras
and isomorphisms $\frac {W^H_n}{W^H_{n}
  \cap \Id^H(A)} \cong \frac {P^H_n}{P^H_{n}
  \cap \Id_{\mathrm{assoc}}^H(A)}$ of $FS_n$-modules
  where $n\in\mathbb N$ and $P^H_n$ is the $FS_n$-module
  of associative $H$-polynomials multilinear in $x_1, x_2, \dots, x_n$.
  In particular, the definitions of codimensions and cocharacters
  do not depend on whether we use $F \lbrace X | H\rbrace$ or $F\langle X | H \rangle$.
Analogous remarks can be made in the case when $A$ is a Lie algebra (see~\cite[Section 1.3]{ASGordienko5}).
\end{remark}

\section{Graded polynomial identities}\label{SectionGradedPI}

Let $T$ be a set and let $F$ be a field. 

Consider the absolutely free non-associative algebra $F\lbrace X^{T\text{-}\mathrm{gr}} \rbrace$
on the disjoint union $$X^{T\text{-}\mathrm{gr}}:=\bigsqcup_{t \in T}X^{(t)}$$
of the sets $X^{(t)} = \{ x^{(t)}_1,
x^{(t)}_2, \dots \}$

We say that $f=f(x^{(t_1)}_{i_1}, \dots, x^{(t_s)}_{i_s})$ is
a \textit{graded polynomial identity} for
 a $T$-graded algebra $A=\bigoplus_{t\in T}
A^{(t)}$
and write $f\equiv 0$
if $f(a^{(t_1)}_1, \dots, a^{(t_s)}_s)=0$
for all $a^{(t_j)}_j \in A^{(t_j)}$, $1 \leqslant j \leqslant s$.
  The set $\Id^{T\text{-}\mathrm{gr}}(A)$ of graded polynomial identities of
   $A$ is
an ideal of $F\lbrace
 X^{T\text{-}\mathrm{gr}}\rbrace$.

\begin{example}\label{ExampleIdGr}
Consider the multiplicative semigroup $T=\mathbb Z_2 = \lbrace \bar 0, \bar 1 \rbrace$
and the $T$-grading $\UT_2(F)=\UT_2(F)^{(\bar 0)}\oplus \UT_2(F)^{(\bar 1)}$
on the algebra $\UT_2(F)$ of upper triangular $2\times 2$ matrices over a field $F$
defined by
 $\UT_2(F)^{(\bar 1)}=\left(
\begin{array}{cc}
F & 0 \\
0 & F
\end{array}
 \right)$ and $\UT_2(F)^{(\bar 0)}=\left(
\begin{array}{cc}
0 & F \\
0 & 0
\end{array}
 \right)$. We have $$[x^{(\bar 1)}, y^{(\bar 1)}]:=x^{(\bar 1)} y^{(\bar 1)} - y^{(\bar 1)} x^{(\bar 1)}
\in \Id^{T\text{-}\mathrm{gr}}(\UT_2(F))$$
and $x^{(\bar 0)} y^{(\bar 0)} 
\in \Id^{T\text{-}\mathrm{gr}}(\UT_2(F))$.
\end{example}

Let
$$W^{T\text{-}\mathrm{gr}}_n := \langle x^{(t_1)}_{\sigma(1)}
x^{(t_2)}_{\sigma(2)}\cdots x^{(t_n)}_{\sigma(n)}
\mid t_i \in T, \sigma\in S_n \rangle_F \subset F \lbrace X^{T\text{-}\mathrm{gr}} \rbrace$$
(with all possible arrangements of brackets), $n \in \mathbb N$.
The number $$c^{T\text{-}\mathrm{gr}}_n(A):=\dim\left(\frac{W^{T\text{-}\mathrm{gr}}_n}{W^{T\text{-}\mathrm{gr}}_n \cap \Id^{T\text{-}\mathrm{gr}}(A)}\right)$$
is called the $n$th \textit{codimension of graded polynomial identities}
or the $n$th \textit{graded codimension} of $A$.

The symmetric group $S_n$  acts
 on the space $\frac {W^{T\text{-}\mathrm{gr}}_n}{W^{T\text{-}\mathrm{gr}}_{n}
  \cap \Id^{T\text{-}\mathrm{gr}}(A)}$
  by permuting the variables within each $X^{(t)}$: $$\sigma x^{(t_1)}_{i_1}\cdots x^{(t_n)}_{i_n}
  := x^{(t_1)}_{\sigma(i_1)}\cdots x^{(t_n)}_{\sigma(i_n)}$$
  for $n\in\mathbb N$, $\sigma \in S_n$, $1\leqslant i_k\leqslant n$, $1\leqslant k \leqslant n$.
   The character $\chi^{T\text{-}\mathrm{gr}}_n(A)$ of the
  $FS_n$-module $\frac {W^{T\text{-}\mathrm{gr}}_n}{W^{T\text{-}\mathrm{gr}}_n
   \cap \Id^{T\text{-}\mathrm{gr}}(A)}$ is
   called the $n$th
  \textit{cocharacter} of graded polynomial identities of $A$.
  If $\ch F = 0$, we can rewrite it as
  a sum $$\chi^{T\text{-}\mathrm{gr}}_n(A)=\sum_{\lambda \vdash n}
   m(A, T\text{-}\mathrm{gr}, \lambda)\chi(\lambda)$$ of
  irreducible characters $\chi(\lambda)$.
  The number $\ell_n^{T\text{-}\mathrm{gr}}(A):=\sum_{\lambda \vdash n}
   m(A, T\text{-}\mathrm{gr}, \lambda)$ is called the $n$th
  \textit{colength} of graded polynomial identities of $A$.

The proposition below provides a relation between the ordinary and the graded codimensions.

\begin{proposition} 
Let $A$ be a $T$-graded algebra over a field $F$ for some set $T$
not necessarily finite.
Then $c_n(A) \leqslant c^{T\text{-}\mathrm{gr}}_n(A)$.
If $T$ is finite, then $c^{T\text{-}\mathrm{gr}}_n(A)
\leqslant |T|^n c_n(A)$ for all $n\in\mathbb N$.
\end{proposition}
\begin{proof}
Let $t_1, \dots, t_n \in T$.
Denote by $W_{t_1, \dots, t_n}$ the vector space of
multilinear non-associative polynomials in $x_1^{(t_1)}, \dots, x_n^{(t_n)}$.
Then $W_n^{T\text{-}\mathrm{gr}}=\bigoplus_{t_1, \dots, t_n \in T}
W_{t_1, \dots, t_n}$. Moreover \begin{equation}\label{EqDecompWnGradt1t2tn}
\frac{W^{T\text{-}\mathrm{gr}}_n}{W^{T\text{-}\mathrm{gr}}_n \cap \Id^{T\text{-}\mathrm{gr}}(A)} \cong 
\bigoplus_{t_1, \dots, t_n \in T}
\frac{W_{t_1, \dots, t_n}}{W_{t_1, \dots, t_n} \cap \Id^{T\text{-}\mathrm{gr}}(A)}
\end{equation} since we can pick out the component corresponding
to a given $W_{t_1, \dots, t_n}$ in a graded polynomial identity
by putting $x^{(t)}_i = 0$ for all $1\leqslant i \leqslant n$
and $t \ne t_i$. Note that~\eqref{EqDecompWnGradt1t2tn} holds over a field of any characteristic.

 Let $\bar f_1, \dots, \bar f_{c_n(A)}$ be 
 a basis in $\frac{W_n}{W_n \cap \Id(A)}$ where $f_i\in W_n$.
Then for every monomial $w=x_{\sigma(1)}\cdots x_{\sigma(n)}$ (with some arrangement of brackets), $\sigma\in S_n$,
there exist $\alpha_{w,i}\in F$ such that $$x_{\sigma(1)}\cdots x_{\sigma(n)} - \sum_{i=1}^{c_n(A)}\alpha_{w,i} f_i(x_1,\dots, x_n) \in \Id(A).$$
For every choice of $t_1, \dots, t_n \in T$ we replace $x_i$ with $x^{\left(t_i\right)}_i$
and get $$x^{\left(t_{\sigma(1)}\right)}_{\sigma(1)}\cdots x^{\left(t_{\sigma(n)}\right)}_{\sigma(n)} - \sum_{i=1}^{c_n(A)}\alpha_{w,i} f_i\left(x^{(t_1)}_1,\dots, x^{(t_n)}_n\right) \in \Id^{T\text{-}\mathrm{gr}}(A)$$
and $$\frac{W^{T\text{-}\mathrm{gr}}_n}{W^{T\text{-}\mathrm{gr}}_n \cap \Id^{T\text{-}\mathrm{gr}}(A)}
=\left\langle \bar f_i\left(x^{(t_1)}_1,\dots, x^{(t_n)}_n\right)\mathrel{\Bigl|}
1\leqslant i \leqslant c_n(A),\  t_1, \dots, t_n \in T\right\rangle_F.$$
This implies the upper bound for $c_n^{T\text{-}\mathrm{gr}}(A)$ in the case when $T$ is finite.

In order to get the lower bound for $c_n^{T\text{-}\mathrm{gr}}(A)$, for a given $n$-tuple $(t_1, \dots, t_n) \in T^n$ we consider the map $\varphi_{t_1, \dots, t_n} \colon W_n \rightarrow \frac{W^{T\text{-}\mathrm{gr}}_n}{W^{T\text{-}\mathrm{gr}}_n \cap \Id^{T\text{-}\mathrm{gr}}(A)}$
where $\varphi_{t_1, \dots, t_n}(f)$ is the image of $f\left(x^{(t_1)}_1, \dots, x^{(t_n)}_n\right)$
in $ \frac{W^{T\text{-}\mathrm{gr}}_n}{W^{T\text{-}\mathrm{gr}}_n \cap \Id^{T\text{-}\mathrm{gr}}(A)}$ for $f=f(x_1, \dots, x_n) \in W_n$.
The multilinearity of $f$ implies that $f(x_1, \dots, x_n) \equiv 0$ is an ordinary polynomial
identity if and only if $$f\left(x^{(t_1)}_1, \dots, x^{(t_n)}_n\right) \equiv 0 $$ is a graded polynomial identity for every $t_1, \dots, t_n \in T$. In other words, $W_n \cap \Id(A)
= \bigcap\limits_{(t_1, \dots, t_n) \in T^n} \ker \varphi_{t_1, \dots, t_n}$.
Since $W_n$ is a finite dimensional vector space, there exists a finite subset $\Lambda \subseteq T^n$
such that $W_n \cap \Id(A)
= \bigcap\limits_{(t_1, \dots, t_n) \in \Lambda} \ker \varphi_{t_1, \dots, t_n}$.

Consider the
  embedding $$W_n \hookrightarrow
 W_n^{T\text{-}\mathrm{gr}}=\bigoplus_{t_1, \dots, t_n \in T}
W_{t_1, \dots, t_n}$$ where the image of $f(x_1, \dots, x_n)\in W_n$ equals $\sum_{(t_1, \dots, t_n) \in \Lambda} f\left(x^{(t_1)}_1, \dots, x^{(t_n)}_n\right)$. 
Then~\eqref{EqDecompWnGradt1t2tn} and our choice of $\Lambda$ imply that the induced map $\frac{W_n}{W_n \cap \Id(A)} \hookrightarrow \frac{W^{T\text{-}\mathrm{gr}}_n}{W^{T\text{-}\mathrm{gr}}_n \cap \Id^{T\text{-}\mathrm{gr}}(A)}$ is an embedding and the lower bound follows.
\end{proof}

The limit  $\PIexp^{T\text{-}\mathrm{gr}}(A):=\lim\limits_{n\rightarrow\infty} \sqrt[n]{c^{T\text{-}\mathrm{gr}}_n(A)}$ (if it exists) is called the \textit{graded PI-exponent} of $A$.

In Example~\ref{ExampleGr} we have shown that each $T$-graded algebra $A$ with a finite support is
an algebra with a generalized $F^T$-action.
The lemma below shows that instead of studying graded codimensions and cocharacters of $A$
we can study codimensions and cocharacters of its polynomial $F^T$-identities.

\begin{lemma}\label{LemmaCnGrCnGenH}
Let $\Gamma \colon A=\bigoplus_{t\in T} A^{(t)}$ be a grading on 
an algebra $A$ over a field $F$ by a set $T$ such that $\supp \Gamma$ is finite. Then $c_n^{T\text{-}\mathrm{gr}}(A)=c_n^{F^T}(A)$ and
$\chi_n^{T\text{-}\mathrm{gr}}(A)=\chi_n^{F^T}(A)$ for all $n\in \mathbb N$.
 If, in addition, $\ch F = 0$, we have $\ell_n^{T\text{-}\mathrm{gr}}(A)=\ell_n^{F^T}(A)$.
\end{lemma}
\begin{proof} 
Let $$\xi \colon F\lbrace X | F^T \rbrace \to F\lbrace X^{T\text{-}\mathrm{gr}} \rbrace$$ be the algebra homomorphism  defined by $\xi(x_i^h) = \sum\limits_{t\in\supp \Gamma} h(t)x^{(t)}_i$, $i\in\mathbb N$, $h\in F^T$. Let $f\in \Id^{F^T}(A)$. Consider an arbitrary homomorphism $\psi \colon  
F\lbrace X^{T\text{-}\mathrm{gr}} \rbrace \to A$
such that $\psi(x^{(t)}_i)\in A^{(t)}$
for all $t\in T$ and $i\in\mathbb N$. Then the algebra homomorphism $\psi\xi \colon F\lbrace X | F^T \rbrace
\to A$ satisfies the condition $$\psi\xi(x_i^h)=\sum\limits_{t\in\supp \Gamma} h(t)\psi\left(x^{(t)}_i\right)=
h\left(\sum\limits_{t\in\supp \Gamma} \psi\left(x^{(t)}_i\right)\right)=h\,\psi\xi(x_i).$$ Thus $\psi\xi(f) =0$ and $\xi(f)\in \Id^{T\text{-}\mathrm{gr}}(A)$. Hence $\xi\left(\Id^{F^T}(A)\right)\subseteq \Id^{T\text{-}\mathrm{gr}}(A)$.
Denote by $$\tilde \xi \colon F\lbrace X | F^T \rbrace/\Id^{F^T}(A) \to F\lbrace X^{T\text{-}\mathrm{gr}} \rbrace/\Id^{T\text{-}\mathrm{gr}}(A)$$ the homomorphism induced by $\xi$.

Let $$\eta \colon F\lbrace X^{T\text{-}\mathrm{gr}} \rbrace \to F\lbrace X | F^T \rbrace$$
be the algebra homomorphism defined by $\eta\left(x^{(t)}_i\right) = x^{h_t}_i$ for all $i\in \mathbb N$
and $t\in T$. Consider an arbitrary graded polynomial identity $f\in \Id^{T\text{-}\mathrm{gr}}(A)$.
Let $\psi \colon  F\lbrace X | F^T \rbrace \to A$ be a homomorphism satisfying the condition
$\psi(x_i^h)=h\psi(x_i)$ for every $i\in\mathbb N$ and $h\in F^T$.
Then for any $i\in\mathbb N$ and $g, t \in T$ we have
$$h_g \psi\eta\left(x^{(t)}_i\right) = h_g\psi(x^{h_t}_i)=h_g h_t \psi(x_i)
=\left\lbrace \begin{array}{lll} 0 & \text{ if } & g\ne t,\\
                              \psi\eta\left(x^{(t)}_i\right) & \text{ if } & g=t. \end{array}\right.$$
 Thus $\psi\eta\left(x^{(t)}_i\right) \in A^{(t)}$.
 Therefore, $\psi\eta(f)=0$ and $\eta(\Id^{T\text{-}\mathrm{gr}}(A)) \subseteq \Id^{F^T}(A)$.
Denote by $\tilde\eta \colon  F\lbrace X^{T\text{-}\mathrm{gr}} \rbrace/\Id^{T\text{-}\mathrm{gr}}(A) \to
F\lbrace X | F^T \rbrace/\Id^{F^T}(A)$ the induced homomorphism.

Below we use the notation $\bar f = f + \Id^{F^T}(A) \in F\lbrace X | F^T \rbrace/\Id^{F^T}(A)$ for $f\in
F\lbrace X | F^T \rbrace$ and  $\bar f = f + \Id^{T\text{-}\mathrm{gr}}(A) \in F\lbrace X^{T\text{-}\mathrm{gr}} \rbrace/\Id^{T\text{-}\mathrm{gr}}(A)$ for $f\in F\lbrace X^{T\text{-}\mathrm{gr}} \rbrace$.
Observe that $$x^h_i - \sum\limits_{t\in\supp \Gamma} h(t) x^{h_t}_i\in \Id^{F^T}(A)$$ for every $h\in F^T$ and $i\in\mathbb N$.
Hence $$\tilde\eta\tilde\xi\left(\bar x^h_i\right)=\tilde\eta\left(
\sum\limits_{t\in\supp \Gamma} h(t) \bar x^{(t)}_i\right)
=\sum\limits_{t\in\supp \Gamma} h(t) \bar x^{h_t}_i = \bar x^h_i$$
for every $h\in F^T$ and $i\in\mathbb N$. 
Thus $\tilde\eta\tilde\xi=\id_{F\lbrace X | F^T \rbrace/\Id^{F^T}(A)}$
since $F\lbrace X | F^T \rbrace/\Id^{F^T}(A)$ is generated by $\bar x^h_i$ where $h\in F^T$ and $i\in\mathbb N$.
Moreover $\tilde\xi\tilde\eta\left(\bar x^{(t)}_i\right)=
\tilde\xi\left(\bar x^{h_t}_i\right)=\bar x^{(t)}_i$ for every $t\in \supp \Gamma$ and $i\in \mathbb N$.
Therefore, $\tilde\xi\tilde\eta=\id_{F\lbrace X^{T\text{-}\mathrm{gr}} \rbrace/\Id^{T\text{-}\mathrm{gr}}(A)}$
and $F\lbrace X^{T\text{-}\mathrm{gr}} \rbrace/\Id^{T\text{-}\mathrm{gr}}(A) \cong F\lbrace X | F^T \rbrace/\Id^{F^T}(A)$
as algebras. The restriction of $\tilde\xi$ provides the isomorphism of the $FS_n$-modules
$\frac{W^{F^T}_n}{W^{F^T}_n \cap \Id^{F^T}(A)}$ and $\frac{W^{T\text{-}\mathrm{gr}}_n}{W^{T\text{-}\mathrm{gr}}_n\cap \Id^{T\text{-}\mathrm{gr}}(A)}$.  Hence $$c^{F^T}_n(A)=\dim \frac{W^{F^T}_n}{W^{F^T}_n \cap \Id^{F^T}(A)}
= \dim\frac{W^{T\text{-}\mathrm{gr}}_n}{W^{T\text{-}\mathrm{gr}}_n\cap \Id^{T\text{-}\mathrm{gr}}(A)}=c^{T\text{-}\mathrm{gr}}_n(A)$$
and  $\chi_n^{T\text{-}\mathrm{gr}}(A)=\chi_n^{F^T}(A)$
for all $n\in \mathbb N$.
 If, in addition, $\ch F = 0$, we have $\ell_n^{T\text{-}\mathrm{gr}}(A)=\ell_n^{F^T}(A)$.
\end{proof}

\begin{remark}
Again, analogously to Remark~\ref{RemarkHAssNAssCodimTheSame}, in the case when $A$ is an associative or Lie algebra, one can use, respectively, free associative or Lie graded algebras, however the graded codimensions will be the same.
\end{remark}

  \section{Upper bound for $H$-colengths}\label{SectionUpperHSimple}

Throughout Sections~\ref{SectionUpperHSimple} and~\ref{SectionHPIexpExistHSimple}
we assume that the characteristic of the base field $F$ is $0$.

In~\cite[Theorem~1]{ZaiMishchGiaIntermediate}, A.~Giambruno, S.\,P.~Mishchenko, and M.\,V.~Zaicev
proved that \begin{equation*}
\ell_n(A)=\sum_{\lambda \vdash n} m(A,\lambda)
\leqslant (\dim A) (n+1)^{(\dim A)^2+\dim A}
\end{equation*}
 for all $n\in \mathbb N$.
 
 It turns out that for $H$-codimensions of finite dimensional algebras with a generalized $H$-action
 we have the same upper bound (Theorem~\ref{TheoremUpperBoundColengthHNAssoc} below).

Let $A$ be a finite dimensional algebra with a generalized $H$-action for some associative algebra $H$ with $1$.

\begin{lemma}\label{LemmaHTensorCommutative} Let $C$ be a unital commutative associative algebra over $F$. Define on $A\otimes C$
the structure of an algebra with a generalized $H$-action by $h(a \otimes c)
:= ha \otimes c$ for $a\in A$ and $c\in C$. Then $\Id^H(A\otimes C)=\Id^H(A)$.
\end{lemma}
\begin{proof}
Since $C$ is unital, $A\otimes C$ contains an $H$-invariant subalgebra isomorphic to $A$
and therefore $\Id^H(A\otimes C) \subseteq \Id^H(A)$. The proof of the converse inclusion is 
completely analogous to the case of associative algebras without an action~\cite[Lemma 1.4.2]{ZaiGia}.
\end{proof}

Let $a_1, \ldots, a_s$ be a basis in $A$. Fix a number $k\in\mathbb N$.
Denote by $F[\xi_{ij} \mid 1\leqslant i \leqslant s,\ 1\leqslant j \leqslant k ]$
the unital algebra of commutative associative polynomials
in the variables $\xi_{ij}$
with coefficients from $F$.
 The algebra $A \otimes F[\xi_{ij} \mid 1\leqslant i \leqslant s,\ 1\leqslant j \leqslant k ]$
 is again an algebra with a generalized $H$-action via $h(a \otimes f):=ha \otimes f$
 for $a\in A$ and $f\in  F[\xi_{ij} \mid 1\leqslant i \leqslant s,\ 1\leqslant j \leqslant k ]$.
  Denote by $\tilde A_k$ the intersection of all $H$-invariant subalgebras of $A \otimes F[\xi_{ij} \mid 1\leqslant i \leqslant s,\ 1\leqslant j \leqslant k ]$ containing the elements $\xi_j := \sum_{i=1}^s a_i \otimes \xi_{ij}$ where $1\leqslant j \leqslant k$. (The symbols $\xi_j$ represent ``generic elements''
  of the algebra $A$.)
  
  \begin{lemma}\label{LemmaHRelativeK}
  Let $f=f(x_1, \ldots, x_k) \in F\lbrace X | H \rbrace$.
  Then $f\in \Id^H(A)$ if and only if $f(\xi_1, \ldots, \xi_k)=0$
  in $\tilde A_k$.
  \end{lemma}
  \begin{proof} Lemma~\ref{LemmaHTensorCommutative} implies $$\Id^H(A)=\Id^H(A\otimes F[\xi_{ij} \mid 1\leqslant i \leqslant s,\ 1\leqslant j \leqslant k ])\subseteq \Id^H(\tilde A_k).$$
  In particular, $f\in \Id^H(A)$ implies $f(\xi_1, \ldots, \xi_k)=0$.
  
  Conversely, suppose $f(\xi_1, \ldots, \xi_k)=0$.
We claim that  $f(b_1, \ldots, b_k)=0$ for all $b_j \in A$. Indeed, $b_j=\sum_{i=1}^s \alpha_{ij} a_i$
for some $\alpha_{ij} \in F$.
  Consider the homomorphism $$\varphi \colon A\otimes F[\xi_{ij} \mid 1\leqslant i \leqslant s,\ 1\leqslant j \leqslant k ] \to A$$
  of algebras and $H$-modules defined by $a \otimes \xi_{ij}\mapsto \alpha_{ij} a$ for all $a\in A$. Then 
  $$f(b_1, \ldots, b_k)=f(\varphi(\xi_1),\ldots, \varphi(\xi_k))=\varphi(f(\xi_1, \ldots, \xi_k))=0$$
  and $f\in \Id^H(A)$.   
  \end{proof}
  
    \begin{lemma}\label{LemmaHRelativeKdim}
  Denote by $R_{kn}$ the linear span in $\tilde A_k$ of all products $(h_1 \xi_{i_1})\cdots (h_n \xi_{i_n})$ where $h_j \in H$ and $1\leqslant i_j \leqslant k$ for $1\leqslant j\leqslant n$.
  Then $\dim R_{kn}\leqslant (\dim A) (n+1)^{k\dim A}$ for all $n\in\mathbb N$.
  \end{lemma}
\begin{proof}
The space $R_{kn} \subseteq A\otimes F[\xi_{ij} \mid 1\leqslant i \leqslant s,\ 1\leqslant j \leqslant k ]$  is a subspace of the linear span of elements $a_\ell \otimes \prod\limits_{\substack{
1\leqslant i \leqslant s,\\ 1\leqslant j \leqslant k}} \xi_{ij}^{s_{ij}}$
where $1\leqslant \ell\leqslant s =\dim A$, $s_{ij}\in\mathbb Z_+$, $\sum\limits_{\substack{
1\leqslant i \leqslant s,\\ 1\leqslant j \leqslant k}} s_{ij} = n$.
The number of such elements does not exceed $(\dim A) (n+1)^{k\dim A}$, and we get the desired upper bound
for $\dim R_{kn}$.
\end{proof}

Now we show that all irreducible $FS_n$-submodules that occur in the decomposition
of $\frac{W^H_n}{W^H_n \cap\, \Id^H(A)}$ with nonzero multiplicities,
correspond to Young diagrams of height less than or equal to $\dim A$.
  
  \begin{lemma}\label{LemmaHStripeTheorem}
    Let $\lambda \vdash n$, $n\in\mathbb N$.
    Suppose $\lambda_{(\dim A)+1} > 0$. Then $m(A, H,\lambda)= 0$.
  \end{lemma}
  \begin{proof}
    It is sufficient to prove that $e^{*}_{T_\lambda} f \in \Id^H(A)$
    for all $f\in W_n^H$.
Fix some basis of $A$.
Since polynomials are multilinear, it is
sufficient to substitute only basis elements.
 Note that
$e^{*}_{T_\lambda} = b_{T_\lambda} a_{T_\lambda}$
where $b_{T_\lambda}$ alternates the variables of each column
of $T_\lambda$. Hence if we make a substitution and $
e^{*}_{T_\lambda} f$ does not vanish, this implies
 that different basis elements
are substituted for the variables of each column.
But if $\lambda_{(\dim A)+1} > 0$, then the length of the first
 column is greater
than $\dim A$. Therefore,
 $e^{*}_{T_\lambda} f \in \Id^H(A)$.
  \end{proof}

  Now we can prove the main result of this section.
  
  \begin{theorem}\label{TheoremUpperBoundColengthHNAssoc}
  Let $A$ be a finite dimensional algebra with a generalized $H$-action for some associative algebra $H$ with $1$ over a field $F$ of characteristic $0$.
  Then $$\ell_n^H(A) \leqslant (\dim A) (n+1)^{(\dim A)^2+\dim A}$$ for all $n\in\mathbb N$.
  \end{theorem}
  \begin{proof}
Fix for each partition $\lambda \vdash n$ a Young tableux $T_\lambda$
of the shape $\lambda$.
Then for $\lambda, \mu \vdash n$ we have $e_{T_\lambda} FS_n e_{T_\mu} = \left\lbrace\begin{array}{ccc}
Fe_{T_\lambda} &\text{ if } & \lambda = \mu,\\
0 & \text{ if } & \lambda \ne \mu.
\end{array} \right.$ (See e.g.~\cite[{Lemma~4.23 and Exercise~4.24}]{FultonHarris}.)
Hence the multiplicity $m(A,H,\lambda)$ of $M(\lambda)=FS_n e_{T_\lambda}$
in $\frac{W^H_n}{W^H_n \cap\, \Id^H(A)}$ equals $\dim e_{T_\lambda} \frac{W^H_n}{W^H_n \cap\, \Id^H(A)}$.
In other words, $m(A,H,\lambda)$ equals the maximal number $m$ of $H$-polynomials $f_1,\ldots, f_m \in W_n^H$ such that $g=\alpha_1 e_{T_\lambda}f_1 + \dots \alpha_m e_{T_\lambda}f_m \in \Id^H(A)$
for some $\alpha_\ell\in F$ always implies $\alpha_1=\ldots = \alpha_m = 0$.
Denote by $k_{ij}$ the number in the $(i,j)$th box of $T_\lambda$.
Then for a fixed $i$ each $e_{T_\lambda}f_\ell$ is symmetric in the variables $ x_{k_{i1}}, \ldots, x_{k_{i\lambda_i}} $.
Applying the linearization procedure (see e.g.~\cite[Section~1.3]{ZaiGia}), we obtain that $g$ is a polynomial $H$-identity if and only if 
$\tilde g$ is a polynomial $H$-identity, where $\tilde g$ is obtained
from $g$ by the substitution $x_{k_{ij}} \mapsto x_i$ for all $i$ and $j$.
Denote the number of rows in $T_\lambda$ by $k$. 
By Lemma~\ref{LemmaHStripeTheorem}, we may assume that $k\leqslant \dim A$.
 The $H$-polynomial $\tilde g$ depends on the variables $x_1,\ldots, x_k$
 and Lemma~\ref{LemmaHRelativeK} implies that $\tilde g \in \Id^H(A)$
 if and only if $\tilde g(\xi_1, \ldots, \xi_k) = 0$ in $\tilde A_k$.
 Note that $\tilde g(\xi_1, \ldots, \xi_k) = \alpha_1  u_1 + \dots + \alpha_m u_m$
 where $u_\ell$ is the value of $e_{T_\lambda}f_\ell$ under the substitution
 $x_{k_{ij}} \mapsto \xi_i$ for $1\leqslant i \leqslant k$ and $1\leqslant j \leqslant \lambda_i$. Hence all $u_i \in R_{kn}$ and if $m >  (\dim A) (n+1)^{k\dim A}$,
 then by Lemma~\ref{LemmaHRelativeKdim} for any choice of $f_i$ the elements $u_i$ are linearly
 dependant and $\tilde g(\xi_1, \ldots, \xi_k) = \alpha_1  u_1 + \dots + \alpha_m u_m = 0$
 for some nontrivial $\alpha_i$. In particular, $\alpha_1 e_{T_\lambda}f_1 + \dots \alpha_m e_{T_\lambda}f_m \in \Id^H(A)$ and $m(A,H,\lambda) < m$. Hence for any $\lambda \vdash n$
 we have $m(A,H,\lambda)\leqslant (\dim A) (n+1)^{k\dim A} \leqslant (\dim A) (n+1)^{(\dim A)^2}$.
 Since the number of all partitions $\lambda \vdash n$ of height not greater than
 $\dim A$ does not exceed $n^{\dim A}$,
 we get the desired upper bound for $\ell_n^H(A)$.
\end{proof}

By Lemma~\ref{LemmaCnGrCnGenH} above, if a finite dimensional algebra $A$
  is graded by a set $T$, then the colengths $\ell_n^{T\text{-}\mathrm{gr}}(A)$
  of graded polynomial identities of $A$ are equal to the $F^T$-colengths
  $\ell_n^{F^T}(A)$. Thus we immediately get the following corollary of Theorem~\ref{TheoremUpperBoundColengthHNAssoc}:
\begin{corollary}\label{CorollaryUpperBoundColengthGraded}
 Let $A$ be a finite dimensional algebra over a field of characteristic $0$
  graded by a set $T$. Then 
  $$\ell_n^{T\text{-}\mathrm{gr}}(A) \leqslant (\dim A) (n+1)^{(\dim A)^2+\dim A}$$ for all $n\in\mathbb N$.
\end{corollary}
  
\section{Existence of the $H$-PI-exponent for $H$-simple algebras}\label{SectionHPIexpExistHSimple}

In Theorem~\ref{TheoremHSimpleHPIexpHNAssoc} below
we prove that for every finite dimensional $H$-simple algebra there exists
an $H$-PI-exponent.
  
Let $\Phi(x_1, \dots, x_s)=\frac{1}{x_1^{x_1} \cdots x_s^{x_s}}$ for $x_1, \dots, x_s > 0$.
Since $\lim_{x\to +0} x^x = 1$, we may assume that $\Phi$ is a continuous function for
$x_1, \dots, x_s \geqslant 0$.
  
   \begin{theorem}\label{TheoremHSimpleHPIexpHNAssoc} Let $A$ be a finite dimensional $H$-simple algebra for some associative algebra $H$ with $1$ over a field $F$ of characteristic $0$, $\dim A = s$. Let $$d(A) := \mathop{\overline\lim}_{n\to\infty}\max\limits_{\substack{\lambda \vdash n, \\ m(A,H,\lambda)\ne 0}}
  \Phi\left(\frac{\lambda_1}{n},\dots, \frac{\lambda_s}{n}\right).$$
  Then there exists $$\PIexp^{H}(A) := \lim_{n\to \infty} \sqrt[n]{c_n^H(A)} = d(A).$$
  \end{theorem}
  
  Theorem~\ref{TheoremHSimpleHPIexpHNAssoc} will be proved below.
 
 Again, combining Theorem~\ref{TheoremHSimpleHPIexpHNAssoc} with Lemma~\ref{LemmaCnGrCnGenH} we get:
  
  \begin{corollary}\label{CorollaryGradedExistsExponent}
  Let $A$ be a finite dimensional algebra over a field of characteristic $0$
  graded by a set $T$ such that $A$ does not have non-trivial graded ideals. Then there exists $\PIexp^{T\text{-}\mathrm{gr}}(A)=\lim_{n\to \infty} \sqrt[n]{c_n^{T\text{-}\mathrm{gr}}(A)}$.
  \end{corollary}
   
  First we prove that the $H$-codimension sequence
  is non-decreasing for any $H$-simple algebra.
  
   \begin{lemma}\label{LemmaCodimNondecrHsimpleHNAssoc} Let $A$ be an $H$-simple algebra for some associative algebra $H$ with $1$ over any field $F$.
  Then $c_n^H(A) \leqslant c_{n+1}^H(A)$ for all $n\in\mathbb N$.
  \end{lemma}
  \begin{proof} Fix some $n\in\mathbb N$.
  Let $f_1(x_1, \dots, x_n)$, \dots, $f_{c_n^H(A)}(x_1, \dots, x_n)$ be such $H$-polynomials that their images form a basis
  in $\frac{W_n^H}{W_n^H \cap \Id^H(A)}$.
  Suppose the $H$-polynomials $f_1(x_1, \dots, x_n x_{n+1}), \dots, f_{c_n^H(A)}(x_1, \dots, x_n x_{n+1})$ are linearly dependent modulo $\Id^H(A)$.
  Then there exist $\alpha_1, \dots, \alpha_{c_n^H(A)} \in F$
  such that $$\alpha_1 f_1(a_1, \dots, a_n a_{n+1})+ \dots + \alpha_{c_n^H(A)} f_{c_n^H(A)}(a_1, \dots, a_n a_{n+1}) = 0$$
  for all $a_i \in A$. Since $A$ is $H$-simple, $AA=A$, and 
  $$\alpha_1 f_1(a_1, \dots, a_n)+ \dots + \alpha_{c_n^H(A)} f_{c_n^H(A)}(a_1, \dots, a_n) = 0$$
  for all $a_i \in A$. However, $f_1(x_1, \dots, x_n), \dots, f_{c_n^H(A)}(x_1, \dots, x_n)$
  are linearly independent modulo $\Id^H(A)$. Hence $\alpha_1= \dots= \alpha_{c_n^H(A)}=0$,
  $f_1(x_1, \dots, x_n x_{n+1}), \dots, f_{c_n^H(A)}(x_1, \dots, x_n x_{n+1})$ are linearly independent modulo $\Id^H(A)$, and $c_n^H(A) \leqslant c_{n+1}^H(A)$.
  \end{proof} 
  
  Next we prove the upper bound for $c_n^H(A)$.
  
  \begin{theorem}\label{TheoremUpperBoundCodimPhiHNAssoc} Let $A$ be a finite dimensional algebra with a generalized $H$-action for some associative algebra $H$ with $1$ over a field $F$ of characteristic $0$, $\dim A = s$.
  Then there exist $C > 0$ and $r\in \mathbb R$ such that
  $$c_n^H(A) \leqslant Cn^r \left(\max\limits_{\substack{\lambda \vdash n, \\ m(A,H,\lambda)\ne 0}}
  \Phi\left(\frac{\lambda_1}{n},\dots, \frac{\lambda_s}{n}\right)\right)^n \text{ for all } n\in\mathbb N.$$
  \end{theorem} 
\begin{proof}
Let $\lambda \vdash n$ such that $m(A,H,\lambda)\ne 0$.
By the hook formula, $\dim M(\lambda)=\frac{n!}{\prod_{i,j} h_{ij}}$
where $h_{ij}$ is the length of the hook with the edge in $(i,j)$
in the Young diagram $D_\lambda$. Hence
$\dim M(\lambda) \leqslant \frac{n!}{\lambda_1! \cdots \lambda_s!}$.
By the Stirling formula, for all sufficiently large $n$ we have \begin{equation}\begin{split}\label{EqMlambdaUpperFd}\dim M(\lambda) \leqslant \frac{C_1 n^{r_1}
\left(\frac{n}{e}\right)^n}{\left(\frac{\lambda_1}{e}\right)^{\lambda_1}\cdots
\left(\frac{\lambda_s}{e}\right)^{\lambda_s}}=C_1 n^{r_1}\left(\frac{1}
{\left(\frac{\lambda_1}{n}\right)^{\frac{\lambda_1}{n}}\cdots
\left(\frac{\lambda_s}{n}\right)^{\frac{\lambda_s}{n}}}\right)^n
\\ \leqslant C_1 n^{r_1} \left(\Phi\left(\frac{\lambda_1}{n}, \dots, \frac{\lambda_s}{n}\right)\right)^n \end{split}\end{equation}
for some $C_1 > 0$ and $r_1 \in\mathbb R$ that do not depend on $\lambda_i$.
Together with Theorem~\ref{TheoremUpperBoundColengthHNAssoc} this yields the theorem.
\end{proof}

Throughout the rest of the section we work under the assumptions of Theorem~\ref{TheoremHSimpleHPIexpHNAssoc}.

Let $\lambda \vdash n$, $\mu \vdash m$, 
 $FS_n \bar f_1 \cong M(\lambda)$ and $FS_m \bar f_2 \cong M(\mu)$
for some $m,n \in\mathbb N$, $f_1 \in W_n^H$ and $f_2 \in W_m^H$.
Then the image of the polynomial $f_1(x_1, \dots, x_n)f_2(x_{n+1}, \dots, x_{m+n})$
generates an $FS_{m+n}$-submodule of $\frac{W^H_{m+n}}{W^H_{m+n} \cap\, \Id^H(A)}$ which
is a homomorphic image of $$M(\lambda) \hatotimes M(\mu) := (M(\lambda) \otimes M(\mu))
\uparrow S_{m+n} := FS_{m+n} \mathbin{\otimes_{F(S_n \times S_m)}} (M(\lambda) \otimes M(\mu)).$$
By the Littlewood~--- Richardson rule, all irreducible components in the decomposition
of $M(\lambda) \hatotimes M(\mu)$ correspond to Young diagrams $D_\nu$ that are obtained from
$D_{\lambda+\mu}$ by pushing some boxes down. By our assumptions,
the height of $D_{\nu}$ cannot be greater than $s=\dim A$. Another remark is that, in the process of
pushing boxes down, the value of $\Phi$ is non-decreasing since
the function $\frac{1}{x^x(\xi-x)^{\xi-x}}$ is increasing as $x \in \left(0; \frac{\xi}2 \right)$
for fixed $0 < \xi \leqslant 1$.

\begin{lemma}\label{LemmaSequenceLambdaLowerHSimpleNAssoc}
There exists a constant $N \in\mathbb N$ such that
for every $\varepsilon > 0$ there exist a number $q \in \mathbb N$,
natural numbers $n_1 < n_2 < n_3 < \dots$ such that $n_{i+1} - n_i \leqslant N+q$,
and partitions $\lambda^{(i)} \vdash n_i$, $m\left(A, H, \lambda^{(i)} \right)\ne 0$
such that $\Phi\left(\frac{\lambda^{(i)}_1}{n_i}, \dots, \frac{\lambda^{(i)}_s}{n_i}\right) \geqslant
d(A)-\varepsilon$ for all $i\in \mathbb N$.
\end{lemma}
\begin{proof} Denote by $B$ the subalgebra of the associative algebra $\End_F(A)$
generated by all operators of left and right multiplication by elements of $A$ as well as
the images of all elements of $H$ considered as operators on $A$. Since $A$ is an $H$-simple algebra,
 $A$ is an
irreducible left $B$-module.
Denote by $\mathcal B$ the set of all operators $u\in B$
of the following form: there exist $a_1, \dots, a_m, \tilde a_1, \dots, \tilde a_{\tilde m}
 \in A$, $m, \tilde m\in \mathbb Z_+$, $h \in H$
 and some arrangement of brackets such that
$ua = a_1 \cdots a_m (h a) \tilde a_1 \cdots \tilde a_{\tilde m}$ for all $a\in A$.
 Using~\eqref{EqGeneralizedHopf},
we can incorporate operators of the $H$-action in the operators of left and right multiplication by elements of $A$ and present each operator from $B$ as
a linear combination of operators from $\mathcal B$. 
Since $\End_F(A)$ is finite dimensional,
we can choose a finite basis $u_1, \ldots, u_{\dim B} \in \mathcal B$ of the vector space $B$.
 Denote by $N$ the maximal number
$2(m + \tilde m)$ among all $u_i$.

Since $A$ is $H$-simple, $A^2\ne 0$ and for every $a,b\in A$, $a\ne 0$, $b\ne 0$, 
we have $A=Ba=Bb$ and $(Ba)(Bb)\ne 0$. The choice of $N$ implies that
there exist some $a_1, \dots, a_m, \tilde a_1, \dots, \tilde a_{\tilde m}, b_1,\dots, b_k,
\tilde b_1, \dots, \tilde b_{\tilde k}
 \in A$, $k,\tilde k, m, \tilde m\in \mathbb Z_+$, $h_1, h_2 \in H$, such that
$$(a_1 \cdots a_m (h_1 a) \tilde a_1 \cdots \tilde a_{\tilde m})(b_1 \cdots b_k (h_2 b) 
\tilde b_1 \cdots \tilde b_{\tilde k})\ne 0$$ (for some arrangements of brackets on the multipliers)
and $k+\tilde k + m +\tilde m \leqslant N$.

Now we choose $q \in\mathbb N$ such that $\Phi\left(\frac{\mu_1}{q}, \dots, \frac{\mu_s}{q}\right) \geqslant d(A)-\varepsilon/2$ and $m(A,H,\mu)\ne 0$ for some $\mu \vdash q$. 
Recall that $\Phi$ is continuous on $[0;1]^s$
and therefore uniformly continuous on $[0;1]^s$
since $[0;1]^s$ is compact.
Since we can take $q$ arbitrarily large,
we may assume also that \begin{equation}\begin{split}\label{EquationPhiLambda}\Phi\left(\frac{i\mu_1+\sum_{j=1}^i d_j}{
 iq+\sum_{j=1}^i d_j}, \frac{i\mu_2}{iq+\sum_{j=1}^i d_j}, \dots, \frac{i\mu_s}{iq+\sum_{j=1}^i d_j}\right) \\ =
\Phi\left(\frac{\frac{\mu_1}{q}+\frac{\sum_{j=1}^i d_j}{iq}}{
1+\frac{\sum_{j=1}^i d_j}{iq}}, \frac{\left(\frac{\mu_2}{q}\right)}{
1+\frac{\sum_{j=1}^i d_j}{iq}}, \dots, \frac{\left(\frac{\mu_s}{q}\right)}{1+\frac{\sum_{j=1}^i d_j}{iq}}\right)
\geqslant d(A)-\varepsilon\end{split}\end{equation} for all $i\in\mathbb N$ and all $0 \leqslant d_i \leqslant N$.

Choose $\tilde f \in W^H_{q} \backslash \Id^H(A)$ such that $FS_q \tilde f \cong M(\mu)$.
Remarks made in the beginning of the proof imply that for some
arrangements of brackets, some $h_1, h_2 \in H$, and some $k, \tilde k, m, \tilde m \geqslant 0$ such that $d_1:=k+ \tilde k + m+ \tilde m
\leqslant N$, we have $$f_1:=\left(y_1 \cdots y_k \tilde f^{h_1}(x_1, \dots, x_q)
\tilde y_1 \cdots \tilde y_{\tilde k} \right)\left( z_1 \cdots z_m \tilde f^{h_2}(\tilde x_1, \dots, \tilde x_q)
\tilde z_1 \cdots \tilde z_{\tilde m}\right) \notin \Id^H(A).$$
(The notation $f^h$ was introduced in Remark~\ref{RemarkHActionOnWHn} above.)

Consider the $FS_{q+k+\tilde k}$-submodule $M$ of $\frac{W^H_{q+k+\tilde k}}{
W^H_{q+k+\tilde k} \cap\,\Id^H(A)}$
generated by the image of $y_1 \cdots y_k \tilde f^{h_1}(x_1, \dots, x_q)
\tilde y_1 \cdots \tilde y_{\tilde k}$.
By Remark~\ref{RemarkHActionOnWHn},
the image of $\tilde f^{h_1}(x_1, \dots, x_q)$ generates an $FS_q$-submodule isomorphic to $M(\mu)$
and therefore 
 $M$ is a homomorphic image of $$M(\mu)\mathrel{\widehat\otimes}
FS_{k+\tilde k} := (M(\mu)\otimes
FS_{k+\tilde k}) \uparrow FS_{q+k+\tilde k}.$$
Since all partitions of $k+\tilde k$ are obtained
from the row of length $k+\tilde k$ by pushing some boxes down,
by the Littlewood~--- Richardson rule,
all the partitions in the decomposition of $M$ are obtained
from $(\mu_1+ k+\tilde k, \mu_2, \dots, \mu_s)$
by pushing some boxes down.
The same arguments can be applied to 
$z_1 \cdots z_m \tilde f^{h_2}(\tilde x_1, \dots, \tilde x_q)
\tilde z_1 \cdots \tilde z_{\tilde m}$. 

Let $n_1 := 2q+d_1$ and
let $\lambda^{(1)}$ be one of the partitions corresponding to
the irreducible components in the decomposition of $FS_{n_1}\bar f_1$.
Then by~(\ref{EquationPhiLambda}), the remarks above and the remark before the lemma, we have $\Phi\left(\frac{\lambda^{(1)}_1}{n_1}, \dots, \frac{\lambda^{(1)}_s}{n_s}\right) \geqslant d(A)-\varepsilon$.

Again, $$f_2:=\left(y_1 \cdots y_k f_1^{h_1}(x_1, \dots, x_q)
\tilde y_1 \cdots \tilde y_{\tilde k} \right)\left( z_1 \cdots z_m \tilde f^{h_2}(\tilde x_1, \dots, \tilde x_q)
\tilde z_1 \cdots \tilde z_{\tilde m}\right) \notin \Id^H(A)$$
for some
arrangements of brackets, some $h_1, h_2 \in H$, and some $k, \tilde k, m, \tilde m \geqslant 0$, $d_2:=k+ \tilde k + m+ \tilde m
\leqslant N$ (maybe different from those for $f_1$). Again, we define $n_2 := 3q+d_1+d_2$.
Denote by $\lambda^{(2)}$ one of the partitions corresponding to
the irreducible components in the decomposition of $FS_{n_2}\bar f_2$.
We continue this procedure and prove the lemma.
\end{proof}

\begin{proof}[Proof of Theorem~\ref{TheoremHSimpleHPIexpHNAssoc}.]
Fix some $\varepsilon > 0$.
Consider $n_i \in \mathbb N$ and $\lambda^{(i)} \vdash n_i$ from Lemma~\ref{LemmaSequenceLambdaLowerHSimpleNAssoc}.
We have 
\begin{equation}\label{EqCnHi(A)LowerPhi}\begin{split} 
c_{n_i}^H(A) \geqslant \dim M(\lambda^{(i)}) = \frac{n_i!}{\prod_{j,k} h_{jk}}
  \geqslant \frac{n_i!}{(\lambda^{(i)}_1+s-1)! \cdots (\lambda^{(i)}_s+s-1)!} \\ \geqslant 
  \frac{n_i!}{n_i^{s(s-1)}\lambda^{(i)}_1! \cdots \lambda^{(i)}_s!} \geqslant
  \frac{C_1 n_i^{r_1} 
\left(\frac{n_i}{e}\right)^{n_i}}{\left(\frac{\lambda^{(i)}_1}{e}\right)^{\lambda^{(i)}_1}\cdots
\left(\frac{\lambda^{(i)}_s}{e}\right)^{\lambda^{(i)}_s}} \\ \geqslant C_1 n_i^{r_1}\left(\frac{1}
{\left(\frac{\lambda^{(i)}_1}{n_i}\right)^{\frac{\lambda^{(i)}_1}{n_i}}\cdots
\left(\frac{\lambda^{(i)}_s}{n_i}\right)^{\frac{\lambda^{(i)}_s}{n_i}}}\right)^{n_i} = C_1 n_i^{r_1}
 \left(\Phi\left(\frac{\lambda^{(i)}_1}{n_i}, \dots, \frac{\lambda^{(i)}_s}{n_i}\right)\right)^{n_i}\end{split}\end{equation}
 for some $C_1 > 0$ and $r_1 \leqslant 0$ which do not depend on $i$.
 
 Let $n \geqslant n_1$. Then $n_i \leqslant n < n_{i+1}$ for some $i\in\mathbb N$.
 Taking into account~\eqref{EqCnHi(A)LowerPhi}, Lemma~\ref{LemmaCodimNondecrHsimpleHNAssoc}
 and the fact that $\Phi(x_1, x_2, \ldots, x_s) \geqslant 1$
 as $0\leqslant x_1,\ldots, x_s \leqslant 1$, we
  get \begin{equation*}\begin{split} \sqrt[n]{c_n^H(A)} \geqslant \sqrt[n]{c_{n_i}^H(A)}\geqslant
  \sqrt[n]{C_1 n_i^{r_1}} \left(\Phi\left(\frac{\lambda^{(i)}_1}{n_i}, \dots, \frac{\lambda^{(i)}_s}{n_i}\right)\right)^\frac{n_i}{n} \\
\geqslant  
     \sqrt[n]{C_1 n^{r_1}}
 \left(\Phi\left(\frac{\lambda^{(i)}_1}{n_i}, \dots, \frac{\lambda^{(i)}_s}{n_i}\right)\right)^{\frac{n-N-q}{n}}  \geqslant  \sqrt[n]{C_1 n^{r_1}}
 \left(d(A)-\varepsilon\right)^{\frac{n-N-q}{n}}.
 \end{split}\end{equation*}
  Hence $\mathop{\underline\lim}_{n\to\infty} \sqrt[n]{c_n^H(A)} \geqslant d(A)-\varepsilon$.
  Since $\varepsilon > 0$ is arbitary, we get $\mathop{\underline\lim}_{n\to\infty}
  \sqrt[n]{c_n^H(A)} \geqslant d(A)$.
  Now Theorem~\ref{TheoremUpperBoundCodimPhiHNAssoc} yields 
   $\lim_{n\to\infty} \sqrt[n]{c_n^H(A)} = d(A)$.
 \end{proof}
 
 \section{Free-forgetful adjunctions corresponding to gradings and generalized $H$-actions}\label{SectionGrHAdjunction}
 
 In this section we analyze the free constructions from Sections~\ref{SectionHPI} and~\ref{SectionGradedPI} from the categorical point
 of view. Here we consider the categories of not necessarily associative algebras, though the analogous adjunctions, of course, exist
 in the case of associative and Lie algebras too.

\subsection{Gradings}\label{SectionGrAdjunction}

Let $T$ be a set and let $F$ be a field.  
Denote by $\mathbf{Vect}^{T\text{-}\mathrm{gr}}_F$ the category where the objects are all \textit{$T$-graded vector spaces} over $F$, i.e. vector spaces
$V$ with a fixed decomposition $V= \bigoplus_{t\in T} V^{(t)}$,
and the sets $\mathbf{Vect}^{T\text{-}\mathrm{gr}}_F(V,W)$ of morphisms
between $V=\bigoplus_{t\in T} V^{(t)}$ and $W=\bigoplus_{t\in T} W^{(t)}$
consist of all linear maps $\varphi \colon V \to W$
such that $\varphi\left(V^{(t)}\right)\subseteq W^{(t)}$
for all $t\in T$.

Denote by $\mathbf{NAAlg}^{T\text{-}\mathrm{pgr}}_F$
(``not necessarily associative partially $T$-graded algebras'')
 the category where the objects are all not necessarily associative algebras $A$ over $F$
with distinguished subspaces $\bigoplus_{t\in T} A^{(t)} \subseteq A$
(the inclusion can be proper) graded by $T$
and if $A \supseteq \bigoplus_{t\in T} A^{(t)}$ and $B\supseteq \bigoplus_{t\in T} B^{(t)}$ are two such objects
then, by the definition, the set $\mathbf{NAAlg}^{T\text{-}\mathrm{pgr}}_F(A,B)$ of morphisms $A\to B$ consists of all algebra homomorphisms $\varphi\colon A\to B$ such that 
$\varphi(A^{(t)})\subseteq B^{(t)}$ for every $t\in T$.

Denote by $U \colon \mathbf{NAAlg}^{T\text{-}\mathrm{pgr}}_F 
\to \mathbf{Vect}^{T\text{-}\mathrm{gr}}_F$
the forgetful functor that assigns to each object $A \supseteq \bigoplus_{t\in T} A^{(t)}$ the $T$-graded vector space $\bigoplus_{t\in T} A^{(t)}$
and restricts homomorphisms to the distinguished subspaces.

Let $V= \bigoplus_{t\in T} V^{(t)}$ be a $T$-graded space.
Let $Y^{(t)}$ be bases in $V^{(t)}$. Denote by $KV$
the absolutely free non-associative algebra $F\lbrace Y\rbrace$
on the basis $Y=\bigsqcup_{t\in T} Y^{(t)}$.
In the basis invariant form, $$KV = \bigoplus_{n=1}^\infty \bigoplus_{\substack{\text{all possible} \\ \text{arrangements} \\
\text{of brackets}}} \underbrace{V \otimes \dots \otimes V}_n$$
and the multiplication is defined
by $vw=v\otimes w$ (the arrangement of brackets in both sides is the same).
We identify $V$ with the corresponding subspace in $KV$
and treat $KV \supseteq V= \bigoplus_{t\in T} V^{(t)}$
as an object of $\mathbf{NAAlg}^{T\text{-}\mathrm{pgr}}_F$.

For each $\varphi \in \mathbf{Vect}^{T\text{-}\mathrm{gr}}_F(V,W)$
 there exists a unique algebra homomorphism $K\varphi \colon KV \to KW$
such that $(K\varphi)\bigl|_{V}=\varphi$.

\begin{proposition}
The functor $K \colon \mathbf{Vect}^{T\text{-}\mathrm{gr}}_F 
\to \mathbf{NAAlg}^{T\text{-}\mathrm{pgr}}_F$ is the left adjoint to $U \colon \mathbf{NAAlg}^{T\text{-}\mathrm{pgr}}_F 
\to \mathbf{Vect}^{T\text{-}\mathrm{gr}}_F$.
\end{proposition}
\begin{proof}
If $V\in \mathbf{Vect}^{T\text{-}\mathrm{gr}}_F$
and $A \in \mathbf{NAAlg}^{T\text{-}\mathrm{pgr}}_F$,
then each morphism $KV \to A$ is uniquely determined
by its restriction to $V$. Hence we obtain
a natural bijection $\mathbf{NAAlg}^{T\text{-}\mathrm{pgr}}_F(KV,A)\to \mathbf{Vect}^{T\text{-}\mathrm{gr}}_F(V, UA)$.
\end{proof}

Suppose now that $V=\bigoplus_{t\in T} V^{(t)}$
where $V^{(t)}$ are the vector spaces with the formal bases $\left(x_i^{(t)}\right)_{i\in\mathbb N}$.
Then $KV$ can be identified with $F\lbrace X^{T\text{-}\mathrm{gr}}\rbrace$
from Section~\ref{SectionGradedPI}.
Every $T$-graded algebra $A$ can be treated as an object of $\mathbf{NAAlg}^{T\text{-}\mathrm{pgr}}_F$
where the subspace $\bigoplus_{t\in T} A^{(t)}$ coincides with $A$.
In this case we have a bijection $\mathbf{NAAlg}^{T\text{-}\mathrm{pgr}}_F(KV,A)\to \mathbf{Vect}^{T\text{-}\mathrm{gr}}_F(V, UA)$
which means that every map $\psi \colon X^{T\text{-}\mathrm{gr}} \to A$,
such that $\psi\left(X^{(t)}\right) \subseteq A^{(t)}$ for each $t\in T$,
can be uniquely extended to an algebra homomorphism $\bar \psi \colon KV \to A$
such that $\bar\psi\left(X^{(t)}\right) \subseteq A^{(t)}$.

\subsection{Generalized $H$-actions}\label{SectionHAdjunction}

Let $H$ be a unital associative algebra over a field $F$.
Denote by ${}_H \mathcal M$ the category of left $H$-modules
and by ${}_H \mathbf{NAAlgSubMod}$ 
(``not necessarily associative algebras with subspaces that are $H$-modules'')
the category where the objects are all not necessarily associative algebras $A$ over $F$ with distinguished subspaces $A_0 \subseteq A$ (the inclusion can be proper),
which are left $H$-modules, and for objects $A \supseteq A_0$
and $B \supseteq B_0$ the set ${}_H \mathbf{NAAlgSubMod}(A,B)$ of morphisms consists of all algebra homomorphisms $\varphi \colon A \to B$ where $\varphi(A_0)\subseteq B_0$ and $\varphi\bigl|_{A_0}$
is a homomorphism of $H$-modules. Here we again have an obvious forgetful functor $U \colon 
{}_H \mathbf{NAAlgSubMod} \to {}_H \mathcal M$ where $UA:=A_0$ and $U\varphi := \varphi\bigl|_{A_0}$.

Let $K$ be a functor ${}_H \mathcal M \to {}_H \mathbf{NAAlgSubMod}$ that assigns to each 
left $H$-module $V$ the absolutely free non-associative algebra $KV:=F\lbrace Z \rbrace$
 where $Z$ is a basis in $V$. 
 In other words, 
 $$KV = \bigoplus_{n=1}^\infty \bigoplus_{\substack{\text{all possible} \\ \text{arrangements} \\
\text{of brackets}}} \underbrace{V \otimes \dots \otimes V}_n$$
and the multiplication is defined
by $vw=v\otimes w$ (the arrangement of brackets in both sides is the same).
We identify $V$ with the corresponding subspace in $KV$
and treat $KV \supseteq V$ as an object of ${}_H \mathbf{NAAlgSubMod}$.
For each $\varphi \in {}_H \mathcal M(V,W)$
 there exists a unique algebra homomorphism $K\varphi \colon KV \to KW$
such that $(K\varphi)\bigl|_{V}=\varphi$.

\begin{proposition}
The functor $K \colon {}_H \mathcal M \to {}_H \mathbf{NAAlgSubMod}$ is the left adjoint to $U \colon 
{}_H \mathbf{NAAlgSubMod} \to {}_H \mathcal M$.
\end{proposition}
\begin{proof}
If $V\in {}_H \mathcal M$
and $A \in {}_H \mathbf{NAAlgSubMod}$,
then each morphism $KV \to A$ is uniquely determined
by its restriction to $V$. Hence we obtain
a natural bijection ${}_H \mathbf{NAAlgSubMod}(KV,A)\to {}_H \mathcal M(V, UA)$.
\end{proof}

Suppose now that $V$ is a free left $H$-module with a formal $H$-basis $Y$, i.e. $V=\bigoplus_{y\in Y} Hy$.
Then $KV$ can be identified with $F\lbrace Y | H \rbrace$
from Section~\ref{SectionHPI}.
Every algebra $A$ with a generalized $H$-action can be treated as an object of 
${}_H \mathbf{NAAlgSubMod}$ where the $H$-module $A_0$ coincides with $A$.
In this case we have a bijection ${}_H \mathbf{NAAlgSubMod}(KV,A)\to {}_H \mathcal M(V, UA)$
which means that every map $\psi \colon Y \to A$
can be uniquely extended to an algebra homomorphism $\bar \psi \colon KV \to A$
such that $\bar\psi\left(h y\right) = h \bar\psi\left( y\right)$
for every $y\in Y$.

\section*{Acknowledgements}

The author is grateful to Yuri Bahturin and Mikhail Zaicev who drew his attention to this topic. In addition, the author appreciates the referee for carefully reading the manuscript
and providing a list of misprints and suggestions.

\end{document}